\documentclass{scrartcl}
\usepackage{amsmath,amssymb}
\usepackage{dsfont,amsthm}
\usepackage{color}

\newcommand{\N}{{\mathds N}}
\newcommand{\R}{{\mathds R}}
\newcommand{\Z}{{\mathds Z}}

\newcommand{\Cov}{{\rm Cov}}

\newtheorem{lemma}{Lemma}[section]

\newtheorem{theorem}{Theorem}[section]

\begin{document}

\pagestyle{myheadings}
\markright{Bootstrap for dependent Hilbert space-valued random variables}

\title{Bootstrap for dependent Hilbert space-valued random variables with application to von Mises statistics}

\author{Herold Dehling\thanks{Fakult\"{a}t f\"{u}r Mathematik, Ruhr-Universit\"{a}t Bochum, 44780 Bochum, Germany}, Olimjon Sh. Sharipov \thanks{Institute of Mathematics, National University of Uzbekistan, 29 Dormon Yoli Str., Tashkent, 100125, Uzbekistan}, Martin Wendler  \thanks{Fakult\"{a}t f\"{u}r Mathematik, Ruhr-Universit\"{a}t Bochum, 44780 Bochum, Germany,  Email address:  Martin.Wendler@rub.de}}
\date{\today}

\maketitle

\begin{abstract}
Statistical methods for functional data are of interest for many applications. In this paper, we prove a central limit theorem for random variables taking their values in a Hilbert space. The random variables are assumed to be weakly dependent in the sense of near epoch dependence, where the underlying process fulfills some mixing conditions. As parametric inference in an infinite dimensional space is difficult, we show that the nonoverlapping block bootstrap is consistent. Furthermore, we show how these results can be used for degenerate von Mises-statistics.
\end{abstract}

\noindent {\itshape keywords:} absolute regularity, near epoch dependence, Hilbert space, block bootstrap; functional time series


\noindent This research was supported by the Collaborative Research Grant SFB 823 \emph{Statistical modelling of nonlinear dynamic processes}.

\section{Introduction and Main Results}

\subsection{Introduction}

In many medical and biological problems, when you are dealing with genomics, transcriptomics and proteomics data, the number of variables may be much larger than the number of subjects and traditional parametric methods cannot be used while in contrast particular nonparametric methods can, see Marozzi \cite{maro}. Imaging methods in medicine like functional magnetic resonance imaging lead to function valued time series, see Lange \cite{lang}, Aston and Kirch \cite{asto}. Furthermore, observations measured on a fine time grid can be often treated as a sequence of observed functions on longer periods instead of a seasonal time series with high resolution. Examples include environmental data, see H\"ormann and Kokoszka \cite{hoer}, or medical data, see Cuevas, Febrero, and Fraiman \cite{cuev}.

The first aim of this paper is to establish a bootstrap method for dependent Hilbert space-valued random variables. Assume that a sequence of Hilbert space-valued random variables $(X_n)_{n\in\Z}$ with mean $\mu$ satisfies a central limit, i.e. for any Borel set $A$ with $P(N\in\partial A)=0$ we have the convergence
\begin{equation}
\bigg| P\Big(\frac{1}{\sqrt{n}}\sum_{i=1}^n(X_i-\mu)\in A\Big)-P(N\in A)\bigg|\rightarrow 0
\end{equation}
as $n\rightarrow\infty$, where $N$ is a centered Gaussian Hilbert space-valued random variable with mean zero and covariance operator $V$.

In functional data analysis in order to make some statistical inferences (construct confidence regions and tests) on an unknown parameter $\mu$ that is asymptotically normal, one needs to calculate probabilities $P(N\in A)$ for different sets $A$. Such probabilities are not easy to calculate even in the case when the covariance operator $V$ is known and the set $A$ has a simple structure. This probability depends on infinite number of eigenvalues and eigenfunctions of the operator $V$. The situation becomes more complicated when $A$ is a ``bad'' Borel set and $V$ is unknown and has to be estimated. Thus unlike the one dimensional case where, in general, one can use both the central limit and as an alternative the bootstrap, in Hilbert space the bootstrap becomes more important.

Consistency of the bootstrap for the sample mean of the independent random variables with values in Banach spaces were established by Gin\'{e} and Zinn \cite{gine}. To the best of our knowledge there is only one paper by Politis and Romano \cite{poli} in which the validity of the stationary bootstrap for dependent Hilbert space valued random variables was proved. This is also stated in review papers by McMurry and Politis \cite{mcmu} and Gon\c{c}alves and Politis \cite{gon2}. Up to now it is an open problem whether the bootstrap methods with fixed block length can be used in Hilbert space. We will establish a strong consistency of nonoverlapping block bootstrap for the sample mean of dependent Hilbert space-valued random variables. We assume that the time series is near epoch dependent on an unobserved underlying process which is absolutely regular. This is a more general model than the strong mixing assumed by Politis and Romano \cite{poli}. Also, their result is restricted to bounded random variables.

The second aim of the paper is a bootstrap for von Mises statistics of dependent observations. Bootstrap for von Mises and U-statistics with nondegenerate kernel were studied by Arcones and Gine \cite{arco}, Dehling and Mikosch \cite{deh3} in the case of independent observations and by Leucht and Neumann \cite{leu2}, \cite{leu3}, Leucht \cite{leuc} in the case of dependent observations. We want to show that the validity of the bootstrap for von Mises and U-statistics with degenerate kernel can be proved using bootstrap for Hilbert space-valued random variables.

The paper is organized as follows: In the next subsection we will formulate the central limit theorem for stationary sequences of near epoch dependent Hilbert space-valued  random variables that will be used in the proof of the next theorem. The central limit theorem for mixing Hilbert space-valued random variables was studied in Kuelbs and Philipp \cite{kuel}, Dehling \cite{deh1}, Maltsev and Ostrovskii \cite{Mats}, Zhurbenko and Zuparov \cite{zhur}. Under near epoch dependence, a central limit theorem was proved by Chen and White, a weak invariance principle was given by Berkes, Horv\'ath, and Rice \cite{berk}. Subsection \ref{sub3} is devoted to the bootstrap for Hilbert space-valued random variables. In this section we will formulate a theorem which establishes the strong consistency of the nonoverlapping block bootstrap for the sample mean of near epoch dependent Hilbert space-valued random variables. In subsection \ref{sub4} we will give a theorem on the validity of the bootstrap for von Mises statistics of near epoch dependent observations. And finally proofs will be given in section \ref{sec3} where we will use preliminary results from section \ref{sec2}.

\subsection{\label{sub2} Central Limit Theorem for Hilbert Space-Valued Functionals of Mixing Random Variables}

Let $H$ be a separable Hilbert space with the inner product $\langle\cdot,\cdot\rangle$ and norm $\left\|\cdot\right\|=\sqrt{\langle\cdot,\cdot\rangle}$. Consider a two-sided, stationary sequence $(\xi_n)_{n\in\Z}$ of random variables with values in a separable measurable space $S$. We say that $(X_n)_{n\in\Z}$ is a functional of $(\xi_n)_{n\in\Z}$ if there exists a measurable function $f:S^{\Z}\rightarrow H$ such that
\begin{equation}
X_n=f\left((\xi_{n+i})_{i\in\N}\right).
\end{equation}
We say that $f$ is a 1-approximating functional (or near epoch dependent) if there exists a sequence $(a_m)_{m\in\N}$ with $a_m\rightarrow 0$ as $m\rightarrow 0$ and for every $m$ a function $f_m:S^{2m+1}\rightarrow H$ such that
\begin{equation}
E\left\|X_0-f_m(\xi_{-m},\ldots,\xi_m)\right\|\leq a_m \ \ \ \text{for all} \ \ m\in\N.\label{line3}
\end{equation}
As convergence in $L_2$ implies convergence in $L_1$, the 1-approximating property is more general than $L_2$ near epoch dependence, which is used more often in the literature. In what follows, we will assume that the sequence $(\xi_n)_{n\in\Z}$ is absolutely regular ($\beta$-mixing). We define the coefficients of absolute regularity $(\beta_m)_{m\in\Z}$ by
\begin{equation}
\beta_m=E\Big[\sup_{A\in\mathcal{F}_{m}^\infty}\left(P(A|\mathcal{F}_{-\infty}^{0})-P(A)\right)\Big],
\end{equation}
where $\mathcal{F}_{a}^b$ is the $\sigma$-field generated by $\xi_{a},\ldots,\xi_b$, and call the sequence $(\xi_n)_{n\in\Z}$ absolutely regular if $\beta_m\rightarrow 0$ as $m\rightarrow\infty$. For more details on absolute regularity, see the book of Bradley \cite{brad}. Approximating functionals of underlying absolutely regular sequences cover many examples of times series, e.g. linear processes or expanding dynamical systems, see Hofbauer and Keller \cite{hofb}.

The first result of this paper is a central limit theorem for approximating functionals of absolutely regular sequences:

\begin{theorem}\label{theo1} Let $(X_n)_{n\in\Z}$ be a 1-approximating functional of a stationary, absolutely regular sequence $(\xi_n)_{n\in\Z}$ and assume that the following conditions hold for some $\delta>0$
\begin{enumerate}
\item $E\left\|X_1\right\|^{2+\delta}<\infty$,
\item $\sum_{m=1}^\infty(a_m)^{\delta/(1+\delta)}<\infty$,
\item $\sum_{m=1}^\infty(\beta_m)^{\delta/(2+\delta)}<\infty$.
\end{enumerate}
Then $(X_n)_{n\in\N}$ satisfies the central limit theorem, i.e. the weak convergence
\begin{equation}
\frac{1}{\sqrt{n}}\sum_{i=1}^n(X_i-\mu)\Rightarrow N_1\label{theo1l1}
\end{equation}
as $n\rightarrow\infty$ where $N_1$ is a $H$-valued Gaussian random variable with $N(0,V)$ distribution with mean 0 and covariance operator $V$ defined by
\begin{equation}
\langle Vx,y\rangle=\sum_{j=-\infty}^\infty E\langle X_0,x\rangle\langle X_j,y\rangle.
\end{equation}
and
\begin{equation}
E\Big\|\frac{1}{\sqrt{n}}\sum_{i=1}^n(X_i-\mu)\Big\|^2\rightarrow E\left\|N_1\right\|^2\label{theo1l2}.
\end{equation}

\end{theorem}

Chen and White \cite{chen} proved central limit theorems for arrays of dependent (including near epoch dependence) Hilbert space-valued random variables. In the stationary case, their Corollary 3.10 is similar to our Theorem \ref{theo1}. They assume strong mixing, which is more general than absolute regularity, but with a faster rate. We assume 1-approximability, which is more general than $L_2$ near epoch dependence used by Chen and White.

\subsection{\label{sub3} Block Bootstrap for Hilbert Space-Valued Random Variables}

Parametric methods in Hilbert spaces are difficult, because even a normal distribution has an infinite dimensional parameter, which is difficult to estimate, especially under dependence. We will use nonoverlapping block bootstrap introduced by Carlstein \cite{carl} and show its consistency. We will draw blocks of length $p$ with $p=p_n\rightarrow\infty$ as $n\rightarrow\infty$ and $p_n/n\rightarrow 0$. Set $k=[n/p]$ (where $[.]$ denotes the integer part of a real number) and
\begin{align}
I_i&=\left(X_{(i-1)p+1},X_{(i-1)p+2},\ldots,X_{ip}\right)\\
B_i&=\left\{(i-1)p+1,(i-1)p+2,\ldots,ip\right\}.
\end{align}
We produce a new bootstrap sample $X^\star_{1},\ldots,X^\star_{kp}$ choosing $k$ times randomly and independently blocks with
\begin{equation}
P\left((X_{(i-1)p+1}^\star,X_{(i-1)p+2}^\star,\ldots,X_{ip}^\star)=I_j\right)=\frac{1}{k}\ \ \ \text{for} \ \ \ i,j=1,\ldots,k.
\end{equation}
As a bootstrap version of the sample mean we take
\begin{equation}
\bar{X}^\star_{n}:=\frac{1}{kp}\sum_{i=1}^{kp}X_i^\star.
\end{equation}
The randomness of the bootstrap variables $X^\star_{1},\ldots,X^\star_{kp}$ has two sources: The sequence $(X_n)_{n\in\N}$ of random variables and the drawing with replacement of the blocks. After enlarging the probability space, we can assume that these two sources of randomness are defined the same probability space. This allows us to speak about the probability and expectation conditional on $X_1,\ldots,X_n$, which we denote with $P^\star$ and $E^\star$. Note that
\begin{equation}
E^\star\bar{X}^\star_{n}=\frac{1}{kp}\sum_{i=1}^{kp}X_i=:\bar{X}_{kp}.
\end{equation}
In the following theorem we establish the strong consistency of the bootstrap for the sample mean:
\begin{theorem}\label{theo2} Let $(X_n)_{n\in\Z}$ be a 1-approximating functional of a stationary, absolutely regular sequence $(\xi_n)_{n\in\Z}$ and assume that the following conditions hold for some $\delta>0$ and $\delta'\in(0,\delta)$
\begin{enumerate}
\item $E\left\|X_1\right\|^{2+\delta}<\infty$,
\item $\sum_{m=1}^\infty(a_m)^{\delta'/(1+\delta')}<\infty$, $\sum_{m=1}^\infty m^{3/2} a_m<\infty$,
\item $\sum_{m=1}^\infty(\beta_m)^{\delta'/(2+\delta')}<\infty$, $\sum_{m=1}^\infty m\beta_m<\infty$.
\end{enumerate}
Furthermore, let the block length $p$ be nondecreasing, $p\rightarrow\infty$, $p=O(n^{1-c_1})$ for some $c_1>0$ and $p_n=p_{2^l}$ for $n=2^{l-1}+1,\ldots,2^l$. Then almost surely as $n\rightarrow\infty$
\begin{equation}
\sqrt{kp}\left(\bar{X}^\star_{n}-\bar{X}_{kp}\right)\Rightarrow^\star N_1,
\end{equation}
where $N_1$ is a $H$-valued Gaussian random variable with $N(0,V)$ distribution with mean 0 and covariance operator $V$ defined in Theorem \ref{theo1}.
\end{theorem}
With $\Rightarrow^\star$, we denote the weak convergence of the conditional distribution given $X_1,\ldots,X_n$. Applied to the Hilbert space $H=\R$, this theorem improves existing results for the bootstrap of real valued near epoch dependent sequences. Sharipov and Wendler \cite{shar} proved the almost sure bootstrap consistency under $(4+\delta)$ moments. Gon\c{c}alves and de Jong \cite{gonc} and Calhoun \cite{calh} assumed $(2+\delta)$ moments, but showed only convergence in probability of the bootstrap distribution estimator.

The bootstrap might be used to test the hypothesis that the expected values of functional data in two different populations are identical. Let $X_1,\ldots,X_n$ and $Y_1,\ldots,Y_n$ be two independent samples (which might show dependence within each sample). By comparing the difference $\|\bar{X}-{\bar{Y}}\|$ of the sample means to the $(1-\alpha)$ quantile of its bootstrapped counterpart $\|(\bar{X}^\star-E^\star\bar{X}^\star)-({\bar{Y}}^\star-E^\star{\bar{Y}}^\star)\|$, we obtain a test that has asymptotically level $\alpha$. This gives an alternative to the classical Hotelling test, which is not well suited for the high dimensional setup, see Marozzi \cite{maro}.

\subsection{\label{sub4} Application to von Mises Statistics}

Methods for Hilbert space valued random variables might also help to analyze nonlinear statistics of real valued data. We first treat the Cram\'er-von Mises-statistic, we will treat general von Mises-statistics later. Let $(X_n)_{n\in\N}$ be a real-valued, stationary, 1-approximating sequence of random variables. To test if the distribution function of $X_i$ equals $F$, one can use the following test statistic
\begin{equation}
V_n:=\int_{\R}\left(F_n(t)-F(t)\right)^2w(t)dt,
\end{equation}
where $F_n(t):=1/n\sum_{i=1}^n\mathds{1}_{\{X_i\leq t\}}$ is the empirical distribution function and $w$ is a positive, bounded weight function with $\int w(t)dt<\infty$. A typical choice for the weight function is the density $f$ under the hypothesis, so that the Cram\'er-von Mises-statistic can be written as $V_n=\int\left(F_n(t)-F(t)\right)^2dF$. Another common choice is $w(t)=[F(t)(1-F(t))]^{-1}$ (Anderson-Darling-test), which places more weight on the tails of the distribution. But this is not covered by our assumptions, because $w$ is unbounded. Let $H$ be the Hilbert space of measurable functions $f$ with $\langle f,f\rangle <\infty$ for the inner product given by
\begin{equation}
\langle f,g\rangle:=\int_{\R}f(t)g(t)w(t)dt.
\end{equation}
Then we have
\begin{equation}
V_n=\left\|F_n-F\right\|^2
\end{equation}
and $F_n$ can be regarded as a sample mean of the $H$-valued random variables $\left(\mathds{1}_{\{X_n\leq \cdot\}}\right)_{n\in\N}$. By the boundedness of $w$, the mapping $x\mapsto \mathds{1}_{\{x\leq \cdot\}}$ is Lipschitz-continuous and so this sequence is also a 1-approximating functional. If the mixing and approximation conditions of Theorem \ref{theo2} hold, we have that
\begin{equation}
\sqrt{n}\left(F_n-F\right)\ \ \ \text{and} \ \ \ \sqrt{pk}\left(F_n^\star-F_{pk}\right)
\end{equation}
with $F_n^\star(t)=1/(pk)\sum_{i=1}^{pk}\mathds{1}_{\{X_i^\star\leq t\}}$ converge almost surely to the same limit distribution in $H$. As the squared norm is a continuous mapping, the limit distributions of $nV_n$ and
\begin{equation}
{pk}V_n^\star:=\int_{\R}\left(\sqrt{pk}(F_n^\star(t)-F_{pk}(t))^2\right)w(t)dt
\end{equation} 
are almost surely the same, so that we can use bootstrap to derive confidence regions and critical values for tests.

Now we will consider general bivariate and degenerate von Mises-statistics ($V$-statistics). Let $h:\R^2\rightarrow\R$ be a symmetric, measurable function. We call
\begin{equation}
V_n:=\frac{1}{n^2}\sum_{i,j=1}^n h(X_i,X_j)
\end{equation}
$V$-statistic with kernel $h$. The kernel and the related $V$-statistic are called degenerate, if $E(h(x,X_i))=0$ for all $x\in\R$. Furthermore, we assume that $h$ is Lipschitz-continuous and positive definite, i.e.
\begin{equation}
\sum_{i,j=1}^m c_ic_jh(x_i,x_j)\geq 0
\end{equation}
for all $c_1,\ldots,c_n,x_1,\ldots,x_n\in\R$. If additionally $Eh(X_0,X_0)<\infty$, then by Sun's version of Mercers theorem \cite{sun} (see also Leucht and Neumann \cite{leu3}), we have under these conditions a representation
\begin{equation}
h(x,y)=\sum_{l=1}^\infty \lambda_l \Phi_l(x)\Phi_l(y)
\end{equation}
for orthonormal eigenfunctions $(\Phi_l)_{l\in\N}$ with the following properties
\begin{itemize}
\item $E\left(h(x,X_0)\Phi_l(X_0)\right)=\lambda_l\Phi_l(x)$
\item $E\Phi_l(X_0)=0$ for all $l\in\N$,
\item $E\Phi_l^2(X_0)=1$ for all $l\in\N$,
\item $E\Phi_{l_1}(X_0)\Phi_{l_2}(X_0)=0$ for all $l_1\neq l_2$,
\item $\lambda_l\geq 0$ for all $l\in\N$,
\item $\sum_{l=1}^\infty \lambda_l<\infty$.
\end{itemize}
We can treat such $V$-statistics in the setting of Hilbert spaces. Let $H$ be the  Hilbert space of real-valued sequences $y=(y_l)_{l\in\N}$ that satisfy $\sum_{l=1}^\infty \lambda_l y_l^2<\infty$ equipped with the inner product
\begin{equation}\label{product}
\langle y,z\rangle:=\sum_{l=1}^\infty \lambda_l y_l z_l.
\end{equation}
We consider the $H$-valued sequence of random variables $\left((\Phi_l(X_n))_{l\in\N}\right)_{n\in \N}$ and observe that
\begin{equation}
V_n=\frac{1}{n^2}\sum_{i,j=1}^n\sum_{l=1}^\infty \lambda_l \Phi_l(X_i)\Phi_l(X_j)=\sum_{l=1}^\infty\lambda_l\bigg(\frac{1}{n}\sum_{i=1}^n\Phi_l(X_i)\bigg)^2=\left\|\frac{1}{n}\sum_{i=1}^n\big(\Phi_l(X_i)\big)_{l\in\N}\right\|^2.
\end{equation}
If the conditions of Theorems \ref{theo1} and \ref{theo2} hold, we can conclude by the continuous mapping theorem that $nV_n$ and its bootstrap version ${pk}V_n^\star$ have the same limit distribution, where the bootstrap version is given by the squared norm of
\begin{equation}
\frac{1}{pk}\sum_{i=1}^{pk}\left(\Phi_l(X_i)\right)^\star_{l\in\N}-E^\star\left[\left(\Phi_l(X_1)\right)^\star_{l\in\N}\right].
\end{equation}
It is clear that drawing blocks in the Hilbert space $H$ and in $\R$ gives the same result, that is $\left(\Phi_l(X_i)\right)^\star_{l\in\N}=\left(\Phi_l(X_i^\star)\right)_{l\in\N}$, so we can write the bootstrapped $V$-statistic as
\begin{multline}
V_n^\star=\sum_{l=1}^\infty\lambda_l\left(\frac{1}{pk}\sum_{i=1}^{pk}\Phi_l(X_i^\star)-\frac{1}{pk}\sum_{i=1}^{pk}\Phi_l(X_i)\right)^2\\
=\frac{1}{(pk)^2}\sum_{i,j=1}^{pk}h(X_i^\star,X_j^\star)-\frac{2}{(pk)^2}\sum_{i,j=1}^{pk}h(X_i^\star,X_j)+\frac{1}{(pk)^2}\sum_{i,j=1}^{pk}h(X_i,X_j).
\end{multline}
As we see in the last line, we do not have to know the eigenvalues $(\lambda_l)_{l\in\N}$ and eigenfunctions $(\Phi_l)_{l\in\N}$ to calculate the bootstrap version $V_n^\star$. Note that this procedure give the distribution of a degenerate $V$-statistics even if the original kernel is not degenerate. For the bootstrap of a nondegenerate $V$-statistic under dependence (respectively the related $U$-statistic), see Dehling and Wendler \cite{dehl} and Sharipov and Wendler \cite{shar}.

We will now give precise conditions for the bootstrap to hold:
\begin{theorem}\label{theo3}
Let $(X_n)_{n\in\Z}$ be a 1-approximating functional of a stationary, absolutely regular sequence $(\xi_n)_{n\in\Z}$, and let $h$ be a Lipschitz-continuous and positive definite kernel function such that for some $\delta>0$ and $\delta'\in(0,\delta)$
\begin{enumerate}
\item $E\left|h(X_0,X_0)\right|^{1+\delta}<\infty$,
\item $\sum_{m=1}^\infty(a_m)^{\delta'/(1+2\delta')}<\infty$, $\sum_{m=1}^\infty m^{3/2} \sqrt{a_m}<\infty$,
\item $\sum_{m=1}^\infty(\beta_m)^{\delta'/(1+\delta')}<\infty$, $\sum_{m=1}^\infty m\beta_m<\infty$.
\end{enumerate}
Furthermore, let the block length $p$ be nondecreasing, $p\rightarrow\infty$, $p=O(n^{1-c_1})$ for some $c_1>0$ and $p_n=p_{2^l}$ for $n=2^{l-1}+1,\ldots,2^l$. Then almost surely $nV$ and $pkV_n^\star$ converge to the same limit in distribution.
\end{theorem}
This is a direct consequence of Theorem \ref{theo2} and Lemmas \ref{lem1a} and \ref{lem1b} below, keeping in mind that $E\left\|(\Phi_l(X_i))_{l\in\N}\right\|^{2+2\delta}=E|h(X_i,X_i)|^{1+\delta}<\infty$. It is easy to see that the Cram\'er-von Mises-statistic is an example of a $V$-statistic satisfying the conditions of Theorem \ref{theo3}. Other examples include the $\mathcal{X}^2$-test statistic for the hypothesis of a given distribution with a finite support.

Leucht and Neumann \cite{leu3} proved a similar theorem for the dependent wild bootstrap, which works the following way: Let $(W_{i,n})_{1\leq n, n\in\N}$ be a weakly dependent, rowwise stationary triangular array of centered, unit variance multipliers, such that the autocorrelation $E(W_{i,n}W_{i+k,n})$ tends to 1 as $n\rightarrow\infty$. As a bootstrap version of a $V$-statistic, they consider
\begin{equation*}
 \tilde{V}:=\sum_{i,j=1}^{n}h(X_i,X_j)(W_{i,n}-\bar{W})(W_{j,n}-\bar{W}).
\end{equation*}
Not only their bootstrap method is different, they also assumed a different form of dependence ($\tau$-dependence instead of 1-approximating functionals) and used different techniques for their proofs.

$U$-statistics $U_n$ are defined similar as $V$-statistics:
\begin{equation}
U_n:=\frac{2}{n(n-1)}\sum_{1\leq i<j\leq n}h(X_i,X_j).
\end{equation}
A short calculation gives $U_n=\frac{n}{(n-1)}V_n-\frac{1}{n(n-1)}\sum_{i=1}^nh(X_i,X_i)$, so it follows that the $U$-statistic and its bootstrap version have the same limit as the $V$-statistic with the same kernel $h$.

\section{\label{sec2} Preliminary Results}

\begin{lemma}\label{lem1a} Let $h$ be a Lipschitz-continuous kernel with constant $L$ and with representation
\begin{equation}
h(x,y)=\sum_{l=1}^\infty \lambda_l \Phi_l(x)\Phi_l(y).
\end{equation}
Then the mapping $x\rightarrow (\Phi_l(x))_{l\in\N}$ into the Hilbert space of sequences equipped with the inner product (\ref{product}) is $1/2$-H\"older-continuous.
\end{lemma}

\begin{proof}
Recall that $h(x,y)=\langle(\Phi_l(x))_{l\in\N},(\Phi_l(y))_{l\in\N}\rangle$. The following short calculation leads to the statement of the lemma:
\begin{multline}
\left\|(\Phi_l(x))_{l\in\N}-(\Phi_l(y))_{l\in\N}\right\|^2=\langle(\Phi_l(x))_{l\in\N}-(\Phi_l(y))_{l\in\N},(\Phi_l(x))_{l\in\N}-(\Phi_l(y))_{l\in\N}\rangle\\
=\sum_{l=1}^\infty \lambda_l \Phi_l(x)\Phi_l(x)-\sum_{l=1}^\infty \lambda_l \Phi_l(x)\Phi_l(y)-\sum_{l=1}^\infty \lambda_l \Phi_l(y)\Phi_l(x)+\sum_{l=1}^\infty \lambda_l \Phi_l(y)\Phi_l(y)\\
=h(x,x)-h(x,y)-h(y,x)+h(y,y)\leq 2L|x-y|.
\end{multline}
\end{proof}

\begin{lemma}\label{lem1b} Let $H$ and $H'$ be Hilbert spaces and $(X_n)_{n\in\Z}$ be a 1-approximating functional with approximation constants $(a_{m})_{m\in\N}$ and $g:H\rightarrow H'$ be a $\alpha$-H\"older-continuous function with constant $L$. Then $(g(X_n))_{n\in\Z}$ is a 1-approximating functional with approximation constants $(L(a_{m})^\alpha)_{m\in\N}$.
\end{lemma}

\begin{proof} Let $f_m$ be functions such that
\begin{equation}
E\left\|X_0-f_m(\xi_{-m},\ldots,\xi_m)\right\|\leq a_m.
\end{equation}
Then
\begin{multline}
E\left\|g(X_0)-g(f_m(\xi_{-m},\ldots,\xi_m))\right\|\leq EL\left\|X_0-f_m(\xi_{-m},\ldots,\xi_m)\right\|^\alpha\\
\leq L \left(E\left\|X_0-f_m(\xi_{-m},\ldots,\xi_m)\right\|\right)^\alpha\leq L(a_{m})^\alpha.
\end{multline}
So the condition in (\ref{line3}) is satisfied with approximating functions $g\circ f_m$ and approximation constants $(L(a_{m})^\alpha)_{m\in\N}$.
\end{proof}

\begin{lemma}\label{lem1c} Let $(X_n)_{n\in\Z}$ be a stationary, 1-approximating functional with constants $(a_m)_{m\in\N}$ of an absolutely regular process $(\xi_n)_{n\in\Z}$ with mixing coefficients $(\beta_m)_{m\in\N}$. Then for $k\in\N$, there exist $H$-valued radom sequences $(X'_n)_{n\in\Z}$ and $(X''_n)_{n\in\Z}$ with the same distribution as $(X_n)_{n\in\Z}$ and a set $A$ with $P(A)\geq 1-\beta_{\lfloor \frac{k}{3}\rfloor}$, such that
\begin{itemize}
\item $(X''_n)_{n\in\Z}$ is independent of $(X_n)_{n\in\Z}$,
\item $E\left[\left\|X_i-X_i'\right\|\mathds{1}_A\right]\leq 2a_{i-\lfloor \frac{2k}{3}\rfloor}$ for all $i\geq k$,
\item $E\left[\left\|X'_i-X_i''\right\|\right]\leq 2a_{i+\lfloor \frac{k}{3}\rfloor}$ for all $i\geq 0$.
\end{itemize}
\end{lemma}

This is Proposition 2.16 of Borovkova, Burton, and Dehling \cite{boro}.

\begin{lemma}\label{lem1} Let $(X_n)_{n\in\Z}$ be a stationary, 1-approximating functional with constants $(a_m)_{m\in\N}$ of an absolutely regular process $(\xi_n)_{n\in\Z}$ with mixing coefficients $(\beta_m)_{m\in\N}$ and $E\left\|X_1\right\|^{2+\delta}<\infty$ for some $\delta>0$. Then
\begin{multline}
\left|E\langle X_i,X_{i+j}\rangle-\langle EX_i,EX_{i+j}\rangle\right|\\
\leq 2\left(E\left\|X_1\right\|^{2+\delta}\right)^{\frac{2}{2+\delta}}\beta_{[j/3]}^{\delta/(2+\delta)}+4K\left(E\left\|X_1\right\|^{2+\delta}\right)^{\frac{1}{1+\delta}}a_{[j/3]}^{\delta/(2+\delta)}.
\end{multline}
If the random variables $(X_n)_{n\in\N}$ are bounded by a constant $K$, then
\begin{equation}
\left|E\langle X_i,X_{i+j}\rangle-\langle EX_i,EX_{i+j}\rangle\right|\leq 2K^2\beta_{[j/3]}+4Ka_{[j/3]}.
\end{equation}

\end{lemma}

\begin{proof}
For real-valued random variables, this is Lemma 2.18 of Borovkova et al.\cite{boro}. The $H$-valued case can be proved in the same way, so we only give the details for the bounded case. Without loss of generality, let $i=0$. Let $(X'_n)_{n\in\Z}$ and $(X''_n)_{n\in\Z}$ be copies of $(X_n)_{n\in\Z}$ as in Lemma \ref{lem1c}. By the independence of the sequences $(X_n)_{n\in\Z}$ and $(X''_n)_{n\in\Z}$, we have
\begin{equation}
E\langle X''_0,X_j\rangle=\langle EX''_0,EX_j\rangle=\langle EX_0,EX_j\rangle.
\end{equation}
So we obtain
\begin{multline}
\left|E\langle X_0,X_{j}\rangle-\langle EX_0,EX_{j}\rangle\right|=\left|E\langle X'_0,X'_{j}\rangle-\langle EX''_0,EX_{
j}\rangle\right|\\
=\left|E\left[\langle X'_0,X'_{j}\rangle-\langle X''_0,X_{j}\rangle \right]\right|=\left|E\left[\langle X'_0,X'_{j}-X_j\rangle+\langle X'_0-X''_0,X_{j}\rangle \right]\right|\displaybreak[0]\\
\leq \left|E\left[\langle X'_0,X'_{j}-X_j\rangle\right]\right|+\left|E\left[\langle X'_0-X''_0,X_{j}\rangle \right]\right|\leq KE\left\|X'_{j}-X_j\right\|+K\left\|X'_0-X''_0\right\|\\
\leq KE\left\|X'_{j}-X_j\right\|\mathds{1}_A+KE\left\|X'_{j}-X_j\right\|\mathds{1}_{A^C}+K\left\|X'_0-X''_0\right\|\\
\leq 2K a_{[j/3]}+2K^2P(A^C)+2K a_{[j/3]}=2K^2\beta_{[j/3]}+4Ka_{[j/3]}.
\end{multline}

\end{proof}

\begin{lemma}\label{lem2} Let $(X_n)_{n\in\Z}$ be a real-valued, stationary, 1-approximating functional with constants $(a_m)_{m\in\N}$ of an absolutely regular process $(\xi_n)_{n\in\Z}$ with mixing coefficients $(\beta_m)_{m\in\N}$ such that $EX_i=0$ and for some $\delta>0$ $E|X_1|^{2+\delta}<\infty$, $\sum_{m=1}^\infty a_m^{\delta/(1+\delta)}<\infty$ and $\sum_{m=1}^\infty \beta_m^{\delta/(2+\delta)}<\infty$. Then
\begin{equation}
\frac{1}{\sqrt{n}}\left(X_n+\ldots+X_n\right)\Rightarrow N(0,\sigma^2)
\end{equation}
with $\sigma^2=\sum_{j=-\infty}^\infty \Cov(X_0,X_j)<\infty$.
\end{lemma}
This Lemma follows from Theorem 8.6.2 of Ibragimov and Linnik \cite{ibra} and Lemma \ref{lem3} below. In the case $\sigma^2=0$, $N(0,\sigma^2)$ shall be understood as the point mass in the origin.

\begin{lemma}\label{lem3} Let $(X_n)_{n\in\Z}$ be a real-valued, 1-approximating functional with constants $(a_m)_{m\in\N}$ of an absolutely regular process $(\xi_n)_{n\in\Z}$ with mixing coefficients $(\beta_m)_{m\in\N}$ such that $EX_i=0$ and for some $\delta>0$ $E|X_1|^{2+\delta}<\infty$. Then
\begin{equation}
\left(E\left|X_n-f_m(\xi_{n-m},\ldots,\xi_{n+m})\right|^{(2+\delta)/(1+\delta)}\right)^{(1+\delta)/(2+\delta)}\leq C a_m^{\delta/(1+\delta)} \ \ \ \text{for all} \ \ m\in\N
\end{equation}
for a constant $C>0$.
\end{lemma}

\begin{proof} We define $Y_m:=X_n-f_m(\xi_{n-m},\ldots,\xi_{n+m})$, so that $E|Y_m|\leq a_m$ and $E|Y_m|^{2+\delta}<\infty$. Consequently by the Markov inequality
\begin{multline}
E|Y_m|^{(2+\delta)/(1+\delta)}\leq E|Y_m|^{(2+\delta)/(1+\delta)}\mathds{1}_{\{|Y_m|\leq a_m^{-1/(1+\delta)}\}}+E|Y_m|^{(2+\delta)/(1+\delta)}\mathds{1}_{\{|Y_m|>a_m^{-1/(1+\delta)}\}}\\
\leq \left(a_m^{-1/(1+\delta)}\right)^{\frac{2+\delta}{1+\delta}-1}E|Y_m|+\left(a_m^{-1/(1+\delta)}\right)^{(2+\delta)-\frac{2+\delta}{1+\delta}-1}E|Y_m|^{2+\delta}\leq Ca_m^{\frac{\delta(2+\delta)}{(1+\delta)^2}}
\end{multline}
and finally $\left(E\left|Y_m\right|^{(2+\delta)/(1+\delta)}\right)^{(1+\delta)/(2+\delta)}\leq C a_m^{\delta/(1+\delta)}$.
\end{proof}

\begin{lemma}\label{lem4} Let $(X_n)_{n\in\N}$ be a stationary sequence of random variables with values in $H$ such that $EX_1=0$, $E\left\|X_1\right\|^2<\infty$. If there exists a constant $C>0$ such that for all $n\in\N$
\begin{equation}
E\left\|X_{1}+X_2+\ldots+X_n\right\|^2\leq Cn,
\end{equation}
then for all $l\in\N$
\begin{equation}\label{line39}
E\max_{n\leq 2^l}\left\|\left(X_{1}+X_2+\ldots+X_n\right)\right\|^2\leq C2^l l^2.
\end{equation}
\end{lemma}

This is a special case of Theorem 3 in M\'oricz \cite{mori} (which also holds in Hilbert spaces). Maximal inequalities of this type were first introduced by Rademacher \cite{rade} and Menchoff \cite{menc}. The next lemma can be found in the paper of Shao and Yu \cite{shao} for real-valued random variables.

\begin{lemma}\label{lem5} Let $(X_n)_{n\in\N}$ be a stationary sequence of random variables with values in $H$ such that $EX_1=0$, $E\left\|X_1\right\|^2<\infty$ and for some $C>0$
\begin{equation}
E\left\|X_{1}+X_2+\ldots+X_n\right\|^2\leq Cn.
\end{equation}
Then almost surely as $n\rightarrow\infty$
\begin{equation}
\frac{1}{\sqrt{n}\log^2n}\left\|X_1+\ldots+X_n\right\|\rightarrow 0.
\end{equation}
\end{lemma}

\begin{proof} First note that for $n\in\left\{2^{l-1}+1,\ldots,2^l\right\}$
\begin{equation}
\frac{1}{\sqrt{n}\log^2 n}\left\|X_1+\ldots+X_n\right\|\leq \frac{1}{2^{\frac{l-1}{2}}(l-1)^2}\max_{m\leq 2^l}\left\|X_1+\ldots+X_m\right\|\label{line42}.
\end{equation}
Using Lemma \ref{lem4}, we get with the help of Chebyshev's inequality
\begin{multline}
\sum_{l=1}^\infty P\left(\frac{1}{2^{\frac{l-1}{2}}(l-1)^2}\max_{m\leq 2^l}\left\|X_1+\ldots+X_m\right\|\geq\epsilon\right)\\
\leq \frac{1}{\epsilon^2}\sum_{l=1}^\infty E\left[\left(\frac{1}{2^{\frac{l-1}{2}}(l-1)^2}\max_{m\leq 2^l}\left\|X_1+\ldots+X_m\right\|\right)^2\right]\leq \frac{C}{\epsilon^2}\frac{2^l l^2}{2^{l-1}l^4}<\infty.
\end{multline}
So with the Borel-Cantelli-lemma, we can conclude that
\begin{equation}
P\left(\frac{1}{2^{\frac{l-1}{2}}(l-1)^2}\max_{m\leq 2^l}\left\|X_1+\ldots+X_m\right\|\geq\epsilon\ \ \text{infinely often}\right)=0.
\end{equation}
That means that the right side of (\ref{line42}) converges to 0 almost surely and the statement of the lemma follows.
\end{proof}

\begin{lemma}\label{lem6} Let $(X_n)_{n\in\Z}$ be a stationary and 1-approximating functional with approximation constants $(a_m)_{m\in\N}$ of an absolutely regular process with mixing coefficients $(\beta_m)_{m\in\N}$. If $X_i$ is bounded by $K$, $EX_1=0$ and
\begin{equation}
\sum_{m=1}^\infty (a_m+\beta_m)<\infty.
\end{equation}
Then
\begin{equation}\label{line46}
E\left\|X_1+X_2+\ldots+X_n\right\|^2\leq C n\left(K^2+\sum_{m=3}^{\lceil\frac{n}{3}\rceil}(K a_m+K^2\beta_m)\right).
\end{equation}
If
\begin{enumerate}
\item $E\left\|X_1\right\|^{2+\delta}<\infty$,
\item $\sum_{m=1}^\infty(a_m)^{\delta/(1+\delta)}<\infty$,
\item $\sum_{m=1}^\infty(\beta_m)^{\delta/(2+\delta)}<\infty$.
\end{enumerate}
then
\begin{equation}
E\left\|X_1+X_2+\ldots+X_n\right\|^2\leq C n\left(\left(E\left\|X_1\right\|^{2+\delta}\right)^{\frac{2}{2+\delta}}+\left(E\left\|X_1\right\|^{2+\delta}\right)^{\frac{1}{1+\delta}}\right)
\end{equation}
\end{lemma}
This lemma can be proved in the same way as Lemma 2.23 of Borovkova et al. \cite{boro}, using our Lemma \ref{lem1} instead of Lemma 2.18 in Borovkova et al. \cite{boro}.

\begin{lemma}\label{lem7}  Let $(X_n)_{n\in\Z}$ be a stationary and 1-approximating functional with approximation constants $(a_m)_{m\in\N}$ of an absolutely regular process with mixing coefficients $(\beta_m)_{m\in\N}$. Assume that $X_i$ is bounded by $K$, $EX_1=0$ and
\begin{equation}
\sum_{m=1}^\infty m(a_m+\beta_m)<\infty.
\end{equation}
Then there exists a constant $C$ such that
\begin{equation}
E\left\|X_1+X_2+\ldots+X_n\right\|^4\leq C K^4n^2.
\end{equation}
\end{lemma}

\begin{proof} First note that by the linearity of the expectation and by the triangle inequality
\begin{multline}
E\left\|X_1+X_2+\ldots+X_n\right\|^4=\sum_{i_1,i_2,i_3,i_4=1}^n E\left[\langle X_{i_1},X_{i_2}\rangle\langle X_{i_3},X_{i_4}\rangle\right]\\
\leq \sum_{i_1,i_2,i_3,i_4=1}^n\left|E\left[\langle X_{i_1},X_{i_2}\rangle\langle X_{i_3},X_{i_4}\rangle\right]\right|.
\end{multline}
We will develop bounds for the summands. In order to keep the proof short, we will concentrate only on case for the ordering of the indices $i_1,i_2,i_3,i_4$. Assume that $i_1<i_2<i_3<i_4$ and $m:=i_2-i_1\geq i_4-i_3$. With the help of Lemma \ref{lem1c}, we find sequences $(X'_n)_{n\in\Z}$ and $(X''_n)_{n\in\Z}$, such that
\begin{itemize}
\item $(X''_n)_{n\in\Z}$ is independent of $(X_n)_{n\in\Z}$,
\item there is a set $A$ with $P(A)\geq 1-\beta_{\lfloor\frac{m}{3}\rfloor}$,
\item $E\left[\left\|X_i-X_i'\right\|\mathds{1}_A\right]\leq 2a_{\lfloor\frac{m}{3}\rfloor}$ for all $i\geq i_2$,
\item $E\left[\left\|X'_i-X_i''\right\|\right]\leq 2a_{(i-i_1)+\lfloor\frac{m}{3}\rfloor}$ for all $i\geq i_1$.
\end{itemize}
Because $(X''_n)_{n\in\Z}$ and $(X_n)_{n\in\Z}$ are independent, we get
\begin{equation}
E\left[\langle X''_{i_1},X_{i_2}\rangle\langle X_{i_3},X_{i_4}\rangle\right]=E\left[\langle EX''_{i_1},X_{i_2}\rangle\langle X_{i_3},X_{i_4}\rangle\right]=0.
\end{equation}
We can conclude that
\begin{multline}
\left|E\left[\langle X_{i_1},X_{i_2}\rangle\langle X_{i_3},X_{i_4}\rangle\right]\right|=\left|E\left[\langle X'_{i_1},X'_{i_2}\rangle\langle X'_{i_3},X'_{i_4}\rangle\right]-E\left[\langle X''_{i_1},X_{i_2}\rangle\langle X_{i_3},X_{i_4}\rangle\right]\right|\\
\leq \left|E\left[\langle X'_{i_1},X'_{i_2}\rangle\langle X'_{i_3},X'_{i_4}\rangle\right]-E\left[\langle X''_{i_1},X_{i_2}\rangle\langle X'_{i_3},X'_{i_4}\rangle\right]\right|\\
+\left|E\left[\langle X''_{i_1},X_{i_2}\rangle\langle X'_{i_3},X'_{i_4}\rangle\right]-E\left[\langle X''_{i_1},X_{i_2}\rangle\langle X_{i_3},X_{i_4}\rangle\right]\right|\displaybreak[0]\\
\leq  \left|E\left[\langle X'_{i_1}-X''_{i_1},X_{i_2}\rangle\langle X'_{i_3},X'_{i_4}\rangle\right]\right|+\left|E\left[\langle X'_{i_1},X'_{i_2}-X_{i_2}\rangle\langle X'_{i_3},X'_{i_4}\rangle\right]\right|\\
+\left|E\left[\langle X''_{i_1},X_{i_2}\rangle\langle X'_{i_3}-X_{i_3},X'_{i_4}\rangle\right]\right|+\left|E\left[\langle X''_{i_1},X_{i_2}\rangle\langle X_{i_3},X'_{i_4}-X_{i_4}\rangle\right]\right|\displaybreak[0]\\
\leq K^2\Big(E\left|\langle X'_{i_1}-X''_{i_1},X_{i_2}\rangle\right|+E\left|\langle X'_{i_1},X'_{i_2}-X_{i_2}\rangle\right|\\
+E\left|\langle X'_{i_3}-X_{i_3},X'_{i_4}\rangle\right|+E\left|\langle X_{i_3},X''_{i_4}-X_{i_4}\rangle\right|\Big)\displaybreak[0]\\\
\leq K^3\Big(E\left\|X'_{i_1}-X''_{i_1}\right\|+E\left\|X'_{i_2}-X_{i_2}\right\|+E\left\|X'_{i_3}-X_{i_3}\right\|+E\left\|X'_{i_4}-X_{i_4}\right\|\Big).
\end{multline}
By Lemma \ref{lem1c}, we know that $E\left\|X'_{i_1}-X''_{i_1}\right\|\leq 2a_{\lfloor\frac{m}{3}\rfloor}$ and that
\begin{multline}
E\left\|X'_{i_2}-X_{i_2}\right\|=E\left\|X'_{i_2}-X_{i_2}\right\|\mathds{1}_A+E\left\|X'_{i_2}-X_{i_2}\right\|\mathds{1}_{A^C}\\
\leq 2a_{\lfloor\frac{m}{3}\rfloor}+KP(A^C)=2a_{\lfloor\frac{m}{3}\rfloor}+K\beta_{\lfloor\frac{m}{3}\rfloor}.
\end{multline}
The same bound holds for the other two summands, so we can conclude that
\begin{equation}
\left|E\left[\langle X_{i_1},X_{i_2}\rangle\langle X_{i_3},X_{i_4}\rangle\right]\right|\leq CK^4\left(a_{\lfloor\frac{m}{3}\rfloor}+\beta_{\lfloor\frac{m}{3}\rfloor}\right).
\end{equation}
Now a short calculation gives
\begin{multline}
\sum_{\substack{1\leq i_1<i_2<i_3<i_4\leq n\\ i_2-i_1=m,i_4-i_3\leq m}}\left|E\left[\langle X_{i_1},X_{i_2}\rangle\langle X_{i_3},X_{i_4}\rangle\right]\right|\leq \sum_{i_1,i_3=1}^n\sum_{i_4=i_3+1}^{i_3+m}\left|E\left[\langle X_{i_1},X_{i_1+m}\rangle\langle X_{i_3},X_{i_4}\rangle\right]\right|\\
\leq \sum_{i_1,i_3=1}^n\sum_{i_4=i_3+1}^{i_3+m}\left(a_{\lfloor\frac{m}{3}\rfloor}+\beta_{\lfloor\frac{m}{3}\rfloor}\right)\leq Cn^2K^4m\left(a_{\lfloor\frac{m}{3}\rfloor}+\beta_{\lfloor\frac{m}{3}\rfloor}\right)
\end{multline}
and consequently
\begin{multline}
\sum_{\substack{1\leq i_1<i_2<i_3<i_4\leq n\\i_4-i_3\leq i_2-i_1}}\left|E\left[\langle X_{i_1},X_{i_2}\rangle\langle X_{i_3},X_{i_4}\rangle\right]\right|\\
=\sum_{m=1}^\infty\sum_{\substack{1\leq i_1<i_2<i_3<i_4\leq n\\ i_2-i_1=m,i_4-i_3\leq m}}\left|E\left[\langle X_{i_1},X_{i_2}\rangle\langle X_{i_3},X_{i_4}\rangle\right]\right|\\
\leq Cn^2K^4\sum_{m=1}^\infty m\left(a_{\lfloor\frac{m}{3}\rfloor}+\beta_{\lfloor\frac{m}{3}\rfloor}\right)=Cn^2K^4.
\end{multline}
Treating the other cases for the ordering of the indices $i_1,i_2,i_3,i_4$ in the same way will lead to the statement of the lemma.
\end{proof}

\begin{lemma}\label{theoa1} Let $X$ be a separable metric space. Then we can construct an equivalent metric in $X$ such that there exists a sequence of bounded uniformly continuous functions $(g_i)_{i\in\N}$ with the following property: for any sequence $\mu_n$ of measures, $\mu_n\Rightarrow\mu$ if and only if for all $i\in\N$
\begin{equation}
\int g_i d\mu_n\rightarrow \int g_i d\mu\ \ \ \text{as} \ \ n\rightarrow\infty.
\end{equation}
\end{lemma}

This follows from the proof of Theorem 3.1 of Varadarajan \cite{vara}.

\begin{lemma}\label{lem8} Let $H_1$ and $H_2$ be Hilbert spaces and $g:H_1\rightarrow H_2$ be uniformly continuous. Then for any $\delta>0$ there exists a Lipschitz-continuous mapping $g_\delta$ (with Lipschitz-constant $L_\delta$ depending on $\delta$) such that
\begin{equation}
\sup_{x\in H_1}\left\|g(x)-g_\delta(x)\right\|\leq \delta.
\end{equation}
\end{lemma}

This is Corollary 2 of Levy and Rice \cite{levy}.

\section{\label{sec3} Proofs of Main Results}

In our proofs, we denote by $C$ a constant which may depend on several parameters (but not on $n\in\N$) and might have different values even in one chain of inequalities.

\begin{proof}[Proof of Theorem \ref{theo1}] Without loss of generality, we can assume that $EX_i=0$. Otherwise replace $X_i$ by $X_i-EX_i$. Set $S_n:=n^{-1/2}(X_1+X_2+\ldots+X_n)$ and note that it is enough to prove that
\begin{enumerate}
\item $\left(\langle a,S_n\rangle\right)_{n\in\N}$ satisfies the central limit theorem in $\R$ for any $a\in H$,
\item for any $\epsilon>0$ there exists a $d\in\N$ and a sequence $(X_{nd})_{n\in\N}$ of random variables taking values in a $d$-dimensional subspace of $H$ such that
\begin{equation}
S_{nd}\Rightarrow N_d
\end{equation}
as $n\rightarrow\infty$, where $N_d$ is a Gaussian random variable and
\begin{equation}
\sup_{n\in\N}E\left\|S_n-S_{nd}\right\|<\epsilon
\end{equation}
with $S_{nd}=n^{-1/2}(X_{1d}+X_{2d}+\ldots+X_{nd})$,
\end{enumerate}
see Ledoux and Talagrand \cite{ledo}. Note that $\left(Y_n\right)_{n\in\N}$ with $Y_i=\langle a,X_i\rangle$ is a real-valued 1-approximating functional of $(\xi_n)_{n\in\Z}$ with approximation constants $(\left\|a\right\|a_m)_{m\in\N}$ such that $EY_1=0$ and $E\left|Y_i\right|^{2+\delta}<\infty$. Lemma \ref{lem2} implies the central limit theorem for $\left(Y_n\right)_{n\in\N}$.

It remains to prove 2. Let $\left\{e_l\big|l\in\N\right\}$ be an orthonormal basis of $H$ so we have the representation
\begin{equation}
S_n=\sum_{l=1}^\infty \langle S_n,e_l\rangle e_l=\sum_{l=1}^\infty S_n^{(l)}e_l.
\end{equation}
with $S_n^{(l)}=n^{-1/2}\sum_{i=1}^n\langle X_i,e_l\rangle$. As a sequence $(X_{nd})_{n\in\N}$ of finite dimensional random variables, we take the $d$-dimensional projections
\begin{align}
X_{nd}&:=\sum_{l=1}^d \langle X_n,e_l\rangle e_l,\\
S_{nd}&:=\frac{1}{\sqrt{n}}\sum_{i=1}^nX_{id}
\end{align}
and we denote the projections on the orthogonal complement as
\begin{align}
\bar{X}_{nd}&:=X_n-X_{nd}\\
\bar{S}_{nd}&:=S_n-S_{nd}=\frac{1}{\sqrt{n}}\sum_{i=1}^n\bar{X}_{id}.
\end{align}
To prove the asymptotic normality of $S_{nd}$, we can use the Cramer-Wold device. By Lemma \ref{lem1b}, we can conclude that that the $\R$-valued sequence given by the linear combination of the coordinates of the random variables $(X_{nd})_{n\in\N}$ form a 1-approximating sequence, so by Lemma \ref{lem2}, we have
\begin{equation}
S_{nd}\Rightarrow N_d\ \ \ \text{as} \ n\rightarrow\infty
\end{equation}
where $N_d$ is a Gaussian random variable with with mean zero and covariance operator $V_d$ defined by
\begin{equation}
\langle V_dx,y\rangle=\sum_{j=-\infty}^\infty E\langle X_{0d},x\rangle \langle X_{jd},y\rangle
\end{equation}
for all $x,y\in H$. It remains to prove that for any positive $\epsilon$ there exists a $d\in\N$ such that we can approximate the partial sum $S_n$ by $S_{nd}$, that means
\begin{equation}
\sup_{n\in\N}E\left\|S_n-S_{nd}\right\|<\epsilon.\label{clt1}
\end{equation}
Using the covariance inequality from Lemma \ref{lem1} and the stationarity of the process, we have the following upper bound for the expectation of this difference:
\begin{multline}
\sup_{n\in\N}E\left\|S_n-S_{nd}\right\|\leq \sup_{n\in\N}\sqrt{E\left\|S_n-S_{nd}\right\|^2}\\
\leq \sup_{n\in\N} \left(E\left\|\bar{X}_{0d}\right\|^2+C_1\left(E\left\|\bar{X}_{0d}\right\|^{2+\delta}\right)^{\frac{2}{2+\delta}}+C_2\left(E\left\|\bar{X}_{0d}\right\|^{2+\delta}\right)^{\frac{1}{1+\delta}}\right)^{1/2}\displaybreak[0]\\
\leq \left(E\left\|\bar{X}_{0d}\right\|^2\right)^{1/2}+\sqrt{C_1}\left(E\left\|\bar{X}_{0d}\right\|^{2+\delta}\right)^{\frac{1}{2+\delta}}+\sqrt{C_2}\left(E\left\|\bar{X}_{0d}\right\|^{2+\delta}\right)^{\frac{1}{2+2\delta}},
\end{multline}
where the constants $C_1=4\sum_{m=1}^\infty\beta_{[m/3]}^{\delta/(2+\delta)}$ and $C_2=8\sum_{m=1}^\infty a_{[m/3]}^{\delta/(1+\delta)}$ do not depend on $n$. Since $E\left\|X_1\right\|^{2+\delta}<\infty$, we can choose for any $\epsilon>0$ a $d\in\N$ so big that (\ref{clt1}) holds, so the proof of (\ref{theo1l1}) (asymptotic normality) is completed. (\ref{theo1l2}) follows in the same way as Lemma 2.23 of Borovkova et al., making use of the stationarity:
\begin{multline}
E\left\|\frac{1}{\sqrt{n}}\sum_{i=1}^n X_i\right\|^2=\frac{1}{n}\sum_{j=-n}^n(n-j)E\langle X_0,X_j\rangle=\sum_{j=-n}^n\frac{n-j}{n}E\langle X_0,X_j\rangle\\
\rightarrow\sum_{j=-\infty}^\infty E\langle X_0,X_j\rangle=\sum_{j=-\infty}^\infty\sum_{l=1}^\infty E\left[\langle X_0,e_l\rangle\langle X_j,e_l\rangle\right]=\sum_{l=1}^\infty \langle Ve_l,e_l\rangle\\
=\sum_{l=1}^\infty E\left[\langle N_1,e_l\rangle\langle N_1,e_l\rangle\right]=E\left\|N_1\right\|^2
\end{multline}
as $n\rightarrow\infty$ by the dominated convergence theorem and the fact that $\sum_{j=-\infty}^\infty|E\langle X_0,X_j\rangle|<\infty$ by Lemma \ref{lem1}.

\end{proof}

\begin{proof}[Proof of Theorem \ref{theo2}] Without loss of generality, we can assume that $EX_i=0$. We introduce the following notation for the normalized bootstrap block sums:
\begin{align}
S_{ni}^\star&:=\frac{1}{\sqrt{p}}\sum_{j\in B_i}\left(X_j^\star-\bar{X}_{kp}\right),\\
S_{ni}&:=\frac{1}{\sqrt{p}}\sum_{j\in B_i}\left(X_j-\bar{X}_{kp}\right).
\end{align}
We will prove the theorem using Mallow's metric which is defined for random variables $X$ and $Y$ with distributions $\mu_X$ respectively $\mu_Y$ as the minimal $L^2$-distance, that is
\begin{equation}
m(\mu_X,\mu_Y):=\inf_{\mu_{(X,Y)}}\left(E\left\|X-Y\right\|^2\right)^{1/2},
\end{equation}
where the infimum is taken over all distributions $\mu_{(X,Y)}$ with marginals $\mu_X$ and $\mu_Y$. For convenience, we will write $m(X,Y)$ instead of $m(\mu_X,\mu_Y)$. Note that convergence in Mallow's metric is equivalent to both convergence in distribution and convergence of the second moments (see Bickel and Freedman \cite{bick}). We will use the following property of Mallow's metric:
\begin{equation}
m^2\left(\frac{1}{\sqrt{n}}\sum_{i=1}^n Z_i,\frac{1}{\sqrt{n}}\sum_{i=1}^n Y_i \right)\leq m^2\left(Z_i,Y_i\right),
\end{equation}
where $(Z_i)_{i\in\N}$ and $(Y_i)_{i\in\N}$ are independent identically distributed random variables. Using this property, we have
\begin{equation}
m^2\left(\sqrt{kp}(\bar{X}_{n}^\star-\bar{X}_{kp}),N_1\right)= m^2\left(\frac{1}{\sqrt{k}}\sum_{i=1}^kS_{ni}^\star,\frac{1}{\sqrt{k}}\sum_{i=1}^k N_i\right)\leq m^2\left(S^\star_{n1},N_1\right),
\end{equation}
where $\left(N_i\right)_{i\in\N}$ is a sequence of independent Gaussian random variables with the distribution $N(0,V)$. Now in order to prove the theorem, it is enough to show that
\begin{equation}
S^{\star}_{n1}\Rightarrow^\star N(0,V)\ \ \ \text{almost surely}\label{boot1}
\end{equation}
and
\begin{equation}
E^\star\left\|S^\star_{n1}\right\|^2\Rightarrow E\left\|N_1\right\|^2\label{boot2}
\end{equation}
almost surely as $n\rightarrow\infty$. First we will prove that almost surely
\begin{equation}
\frac{1}{k}\sum_{i=1}^kh(S_{ni})=E^\star\left[h(S_{n1}^\star)\right]\rightarrow Eh(N_1)=\int_H hd{\mu_{N_1}}\ \ \ \text{as} \ \ n\rightarrow \infty
\end{equation}
for any bounded uniformly continuous function $h:H\rightarrow\R$. We define the modulus of continuity $\epsilon(\delta)$ of $h$ in the usual way as
\begin{equation}
\epsilon(\delta):=\sup_{x,y:\ \left\|x-y\right\|\leq \delta}\left|h(x)-h(y)\right|.
\end{equation}
As $h$ is uniformly continuous, $\epsilon(\delta)\rightarrow 0$ as $\delta\rightarrow 0$. By the triangle inequality, we have that
\begin{multline}
\left|\frac{1}{k}\sum_{i=1}^kh(S_{ni})-Eh(N_1)\right|\\
= \bigg|\frac{1}{k}\sum_{i=1}^k\bigg(h(\frac{1}{\sqrt{p}}\sum_{j\in B_i}(X_j-\bar{X}_{kp}))-h(\frac{1}{\sqrt{p}}\sum_{j\in B_i}X_j)\bigg)\\
+\frac{1}{k}\sum_{i=1}^kh(\frac{1}{\sqrt{p}}\sum_{j\in B_i}X_j)-E(h(\frac{1}{\sqrt{p}}\sum_{j\in B_1}X_j))+E(h(\frac{1}{\sqrt{p}}\sum_{j\in B_1}X_j))-Eh(N_1)\bigg|\displaybreak[0]\\
\leq \frac{1}{k}\left|\sum_{i=1}^kh(\frac{1}{\sqrt{p}}\sum_{j\in B_i}(X_j-\bar{X}_{kp}))-h(\frac{1}{\sqrt{p}}\sum_{j\in B_i}X_j)\right|\\
+\frac{1}{k}\left|\sum_{i=1}^kh(\frac{1}{\sqrt{p}}\sum_{j\in B_i}X_j)-E(h(\frac{1}{\sqrt{p}}\sum_{j\in B_1}X_j))\right|\\
+\left|E(h(\frac{1}{\sqrt{p}}\sum_{j\in B_1}X_j))-Eh(N_1)\right|=I_n+I\! I_n+I\! I\! I_n
\end{multline}
We will treat this three summands separately. For the last summand, note that by Theorem \ref{theo1}, we have asymptotic normality of $\frac{1}{\sqrt{p}}\sum_{j\in B_1}X_j$, so $I\! I\! I_n\rightarrow 0$ as $n\rightarrow \infty$. Concerning $I_n$ note that
\begin{equation}
I_n\leq \epsilon\left(\sqrt{p}\left\|\bar{X}_{kp}\right\|\right)
\end{equation}
(with modulus of continuity $\epsilon$) and by Lemma \ref{lem5} the strong law of large numbers holds that implies $I_n\rightarrow 0$. It remains to prove that almost surely $I\! I_n\rightarrow 0$. We will use the fact that the function $h$ can be approximated by a Lipschitz-continuous function $h_\delta$ with Lipschitz-constant $L=L_{\delta}$ in such a way that for any $\delta>0$
\begin{equation}
\sup_{x\in H}\left|h(x)-h_\delta(x)\right|\leq\delta,
\end{equation}
see Lemma \ref{lem8}. We conclude that we have the following bound for the second summand:
\begin{equation}
I\! I_n\leq 2\delta+\frac{1}{k}\left|\sum_{i=1}^kh_\delta(\frac{1}{\sqrt{p}}\sum_{j\in B_i}X_j)-E(h_\delta(\frac{1}{\sqrt{p}}\sum_{j\in B_1}X_j))\right|=2\delta +I\! I'_n
\end{equation}
Note that the sequence $(\eta_i)_{i\in\N}$ with $\eta_i=h_\delta(\frac{1}{\sqrt{p}}\sum_{j\in B_i}X_i)$ is an approximating functional of the absolutely regular sequence $\left((\xi_j)_{j\in B_i}\right)_{i\in\Z}$ with mixing coefficients
\begin{equation}
\tilde{\beta}_m=\beta_{(m-1)p}.
\end{equation}
Because of the Lipschitz-continuity of $h_\delta$, the sequence $(\eta_i)_{i\in\N}$ has the approximation constants
\begin{equation}
\tilde{a}_m=L\sqrt{p}a_{(m-1)p}.
\end{equation}
Using the fact that $h_\delta$ is bounded by some constant $K$ and (\ref{line46}), we obtain for $k_1<k_2\leq k_{2^l}$
\begin{multline}
E\left(\sum_{i=k_1}^{k_2}\left(h_\delta(\frac{1}{\sqrt{p}}\sum_{j\in B_i}X_j)-E(h_\delta(\frac{1}{\sqrt{p}}\sum_{j\in B_1}X_j))\right)\right)^2\\
\leq C(k_2-k_1)\left(K^2+\sum_{m=3}^{\lceil\frac{k_2-k_1}{3}\rceil}KL\sqrt{p}a_{(m-1)p}+\sum_{m=3}^{\lceil\frac{k_2-k_1}{3}\rceil}K^2\beta_{(m-1)p}\right)\leq C(k_2-k_1)
\end{multline}
as $\sum_{m=3}^{\infty}\sqrt{p}a_{(m-1)p}<\sum_{m=1}^\infty m a_m<\infty$. So the assumptions of Lemma \ref{lem4} hold and we can apply (\ref{line39}) to obtain
\begin{multline}
P\left(\max_{n=2^{l-1}+1,\ldots,2^l}I\! I'_n>\epsilon\right)\\
=P\left(\max_{n=2^{l-1}+1,\ldots,2^l}\frac{1}{k}\left|\sum_{i=1}^kh_\delta(\frac{1}{\sqrt{p}}\sum_{j\in B_i}X_j)-kE(h_\delta(\frac{1}{\sqrt{p}}\sum_{j\in B_1}X_j))\right|>\epsilon\right)\displaybreak[0]\\
\leq \frac{1}{k_{2^{l-1}}^2\epsilon^2}E\left(\max_{n=2^{l-1}+1,\ldots 2^l}\sum_{i=1}^k h_\delta(\frac{1}{\sqrt{p}}\sum_{j\in B_i}X_j)-kE(h_\delta(\frac{1}{\sqrt{p}}\sum_{j\in B_1}X_j))\right)^2\\
\leq C \frac{1}{k_{2^{l-1}}^2\epsilon^2}k_{2^l}l^2.
\end{multline}
Now by our assumptions on the block length, $p=O(n^{1-c_1})$, so $k_{2^l}\leq 2k_{2^{l-1}}$ and $k_{2^{l-1}}\geq C 2^{lc_1}$ and $\sum_{l=1}^\infty P\left(\max_{n=2^{l-1}+1,\ldots,2^l}I\! I'_n>\epsilon\right)<\infty$ and the Borel-Cantelli-Lemma implies that $P\left(\max_{n=2^{l-1}+1,\ldots,2^l}I\! I'_n>\epsilon \ \ \text{infinitely often}\right)=0$. We can conclude that $I\! I'_n\rightarrow 0$ almost surely as $n\rightarrow\infty$. Thus we have proved that for any bounded, uniformly continuous $h$ almost surely
\begin{equation}
E^\star\left[h(S^\star_{n,1})\right]\rightarrow E\left[h(N_1)\right]\ \ \ \text{as}\ \ n\rightarrow\infty.\label{boot4}
\end{equation}
We can find a countable set of bounded and uniformly continuous functions $(f_i)_{i\in\N}$ with $f_i: H\rightarrow \R$ such that the properties of Lemma \ref{theoa1} hold. Then for all $i\in\N$, $f_i$ satisfies (\ref{boot4}) almost surely, that means there exists a set $N_{f_i}$ with $P(N_{f_i})=0$ and  $E^\star[f_i(S^\star_{n,1})]\rightarrow E[f_i(N_1)]$ for all $\omega\in\Omega\setminus N_{f_i}$. We set $N=\bigcup_{i=1}^\infty N_{f_i}$ and observe that $P(N)=0$ and for all $\omega\in\Omega\setminus N$
\begin{equation}
E^\star\left[f_i(S^\star_{n,1})\right]\rightarrow E\left[f_i(N_1)\right]\ \ \ \text{as}\ \ n\rightarrow\infty.
\end{equation}
for each $i\in\N$. Thus Lemma \ref{theoa1} implies (\ref{boot1}). Now we will prove (\ref{boot2}). First note that by the construction of the bootstrap sample
\begin{multline}
E^\star\left\|S^\star_{n1}\right\|^2=\frac{1}{kp}\sum_{i=1}^k\bigg\|\sum_{j\in B_i}(X_j-\bar{X}_{kp})\bigg\|^2\displaybreak[0]\\
=\frac{1}{kp}\sum_{i=1}^k\bigg(\bigg\|\sum_{j\in B_i}X_j\bigg\|^2-E\bigg\|\sum_{j\in B_i}X_j\bigg\|^2\bigg)+\frac{1}{p}E\bigg\|\sum_{j\in B_1}X_j\bigg\|^2-p\left\|\bar{X}_{kp}\right\|^2\\
=\tilde{I}_n+\tilde{I\! I}_n+\tilde{I\! I\! I}_n.
\end{multline}
Theorem \ref{theo1} implies for the second summand
\begin{equation}
\tilde{I\! I}_n=\frac{1}{p}E\bigg\|\sum_{j\in B_1}X_j\bigg\|^2\rightarrow E\left\|N_1\right\|^2
\end{equation}
as $n\rightarrow\infty$, and the strong law of large numbers (Lemma \ref{lem5}) implies the convergence of the last summand, as
\begin{equation}
\tilde{I\! I\! I}_n=p\left\|\bar{X}_{kp}\right\|^2\leq C\bigg(n^{-\frac{1}{2}-\frac{c_1}{2}}\Big\|\sum_{i=1}^{kp}X_i\Big\|\bigg)^2\rightarrow 0 \ \ \ \text{as} \ \ n\rightarrow\infty
\end{equation}
almost surely. It remains to prove the almost sure convergence of the last part:
\begin{equation}
\tilde{I}_n=\frac{1}{kp}\sum_{i=1}^k\bigg(\bigg\|\sum_{j\in B_i}X_j\bigg\|^2-E\bigg\|\sum_{j\in B_i}X_j\bigg\|^2\bigg)\rightarrow 0\ \ \ \text{as} \ \ n\rightarrow\infty.
\end{equation}
As we want to make use of results for bounded random sequences, we have to truncate the random variables $(X_i)_{i\in\N}$. We define for $K>0$ the 1-Lipschitz-continuous trimming function $\Phi_K:H\rightarrow H$ with
\begin{equation}
\Phi_K(x)=\begin{cases} x \ \ \text{for} \ \ \left\|x\right\|\leq K\\ \frac{Kx}{\left\|x\right\|} \ \ \text{for} \ \ \left\|x\right\| >K.\end{cases}
\end{equation}
We will choose $K=K_l=2^{\eta l}$ for an $\eta>0$ to be defined later. Let $\tilde{\Phi}_K(x):=\Phi_K(x)-E\Phi(X_1)$. As $\delta'<\delta$ by our assumptions, we have
\begin{multline}
E\left\|X_i-\tilde{\Phi}_K(X_i)\right\|^{2+\delta'}
\leq 
2^{1+\delta'}E\left\|X_i-\Phi_K(X_i)\right\|^{2+\delta'}
+2^{1+\delta'}\left\|E(X_i-\Phi_K(X_i))\right\|^{2+\delta'}\\
\leq CE\left\|X_i-\Phi_K(X_i)\right\|^{2+\delta'}\leq C K^{\delta'-\delta}E\left\|X_i\right\|^{2+\delta}\leq CK^{\delta'-\delta}.
\end{multline}
Furthermore, the bounded sequence $(\Phi_K(X_n))_{n\in\N}$ is 1-approximating with the same approximation constants as the original sequence $(X_n)_{n\in\N}$ because of the Lipschitz-continuity of $\Phi_K$. So by Lemma \ref{lem6}
\begin{equation}
E\bigg\|\sum_{j\in B_i}\left(X_j-\tilde{\Phi}_K(X_j)\right)\bigg\|^2\leq CpK^{\frac{\delta'-\delta}{1+\delta}}\label{boot3}
\end{equation}
and we can conclude with the help of Lemma \ref{lem4} that
\begin{multline}
E\bigg(\max_{n=2^{l-1}+1,\ldots,2^l}\frac{1}{kp}\sum_{i=1}^k\bigg\|\sum_{j\in B_i}\left(X_j-\tilde{\Phi}_K(X_j)\right)\bigg\|^2\bigg)\\
\leq \frac{1}{k_{2^{l-1}}}\sum_{i=1}^{k_{2^l}}\frac{1}{p}E\bigg\|\sum_{j\in B_i}\left(X_j-\tilde{\Phi}_K(X_j)\right)\bigg\|^2\leq CK^{\frac{\delta'-\delta}{1+\delta}}=C2^{l\eta \frac{\delta'-\delta}{1+\delta}}.
\end{multline}
Note that $\eta \frac{\delta'-\delta}{1+\delta}<0$, so
\begin{equation}
\sum_{l=1}^\infty P\bigg(\max_{n=2^{l-1}+1,\ldots,2^l}\frac{1}{kp}\sum_{i=1}^k\bigg\|\sum_{j\in B_i}\left(X_j-\tilde{\Phi}_K(X_j)\right)\bigg\|^2>\epsilon\bigg)\leq \sum_{l=1}^\infty \frac{C}{\epsilon^2}2^{l\eta \frac{\delta'-\delta}{1+\delta}}<\infty
\end{equation}
and the Borel-Cantelli-Lemma implies that $\frac{1}{kp}\sum_{i=1}^k\left\|\sum_{j\in B_i}\left(X_j-\tilde{\Phi}_K(X_j)\right)\right\|^2\rightarrow 0$ almost surely and consequently
\begin{multline}
\left|\sqrt{E^\star\bigg\|S_{n1}^\star\bigg\|^2}-\sqrt{\frac{1}{kp}\sum_{i=1}^k\bigg\|\sum_{j\in B_i}\tilde{\Phi}_K(X_j)\bigg\|^2}\right|\\
=\left|\sqrt{\frac{1}{kp}\sum_{i=1}^k\bigg\|\sum_{j\in B_i}X_j\bigg\|^2}-\sqrt{\frac{1}{kp}\sum_{i=1}^k\bigg\|\sum_{j\in B_i}\tilde{\Phi}_K(X_j)\bigg\|^2}\right|\\
\leq \sqrt{\frac{1}{kp}\sum_{i=1}^k\bigg\|\sum_{j\in B_i}X_j-\sum_{j\in B_i}\tilde{\Phi}_K(X_j)\bigg\|^2}\rightarrow 0
\end{multline}
almost surely and thus
\begin{equation}
E^\star\left\|S_{n1}^\star\right\|^2-\frac{1}{kp}\sum_{i=1}^k\bigg\|\sum_{j\in B_i}\tilde{\Phi}_K(X_j)\bigg\|^2\rightarrow 0.
\end{equation}
We obtain
\begin{multline}
\left|\sqrt{E\frac{1}{p}\bigg\|\sum_{j\in B_1}\tilde{\Phi}_K(X_j)\bigg\|^2}-\sqrt{E\frac{1}{p}\bigg\|\sum_{j\in B_1}X_j\bigg\|^2}\right|
\leq \sqrt{E\frac{1}{p}\bigg\|\sum_{j\in B_1}\tilde{\Phi}_K(X_j)-X_j\bigg\|^2}\rightarrow 0
\end{multline}
using (\ref{boot3}). So instead of proving $\tilde{I}_n\rightarrow 0$, it suffices to show that
\begin{equation}
\frac{1}{kp}\sum_{i=1}^k\bigg\|\sum_{j\in B_i}\tilde{\Phi}_K(X_j)\bigg\|^2-E\frac{1}{kp}\sum_{i=1}^k\bigg\|\sum_{j\in B_i}\tilde{\Phi}_K(X_j)\bigg\|^2\rightarrow 0
\end{equation}
almost surely. The sequence $(\frac{1}{\sqrt{p}}\sum_{j\in B_i}\tilde{\Phi}_K(X_j))_{i\in\Z}$ is an approximating functional of the absolutely regular sequence $\left((\xi_j)_{j\in B_i}\right)_{i\in\Z}$ with approximating constants $a'_m=\sqrt{p}a_{(m-1)p}$.

The random variables $\left\|\frac{1}{\sqrt{p}}\sum_{j\in B_i}\tilde{\Phi}_K(X_j)\right\|^2$ are bounded by $4pK^2$. As the mapping $x\rightarrow\left\|x\right\|^2$ is Lipschitz continuous with constant $8pK^2$ for arguments bounded by $2pK^2$, it follows that $\Big(\big\|\frac{1}{\sqrt{p}}\sum_{j\in B_i}\tilde{\Phi}_K(X_j)\big\|^2\Big)_{i\in\N}$ forms a sequence of approximating functionals with mixing coefficients
\begin{equation}
\bar{\beta}_k=\beta_{(k-1)p}
\end{equation}
and approximation constants
\begin{equation}
\bar{a}_k=8p^{\frac{3}{2}}K^2 a_{(k-1)p}.
\end{equation}
By Lemma \ref{lem7}, $E\left\|\sum_{j\in B_i}\tilde{\Phi}_K(X_j)\right\|^4\leq CK^4p^2$, and we now obtain with the help of Lemma \ref{lem6} for $k_1<k_2\leq k_{2^l}$
\begin{multline}
E\left(\sum_{i=k_1}^{k_2}\left(\frac{1}{p}\bigg\|\sum_{j\in B_i}\tilde{\Phi}_K(X_j)\bigg\|^2 -\frac{1}{p}E\bigg\|\sum_{j\in B_1}\tilde{\Phi}_K(X_j)\bigg\|^2\right)\right)^2\\
\leq C(k_2-k_1)\left(\frac{1}{p^2}E\bigg\|\sum_{j\in B_1}\tilde{\Phi}_K(X_j)\bigg\|^4+\sum_{m=3}^{\lceil (k_2-k_1)/3\rceil}\left(pK^2\bar{a}_m+p^2K^4\bar{\beta}_m\right)\right)\displaybreak[0]\\
\leq C(k_2-k_1)K^4\left(1+\sum_{m=3}^{\infty}p^{\frac{5}{2}}a_{(m-1)p}+\sum_{m=3}^{\infty}p^2\beta_{(m-1)p}\right)\leq C(k_2-k_1)K^4,
\end{multline}
as $\sum_{m=3}^{\infty}p^{\frac{5}{2}}a_{(m-1)p}\leq \sum_{m=1}^{\infty}m^{3/2}a_{m}<\infty$ and $\sum_{m=3}^{\infty}p^2\beta_{(m-1)p}\leq \sum_{m=1}^{\infty}m\beta_{m}<\infty$. With the help of this moment bound and Lemma \ref{lem4} we get
\begin{multline}
\sum_{l=1}^\infty P\left(\max_{n=2^{l-1}+1,\ldots,2^{l}}\bigg|\frac{1}{k}\sum_{i=1}^{k}\bigg(\frac{1}{p}\bigg\|\sum_{j\in B_i}\tilde{\Phi}_K(X_j)\bigg\|^2 -\frac{1}{p}E\bigg\|\sum_{j\in B_1}\tilde{\Phi}_K(X_j)\bigg\|^2\bigg)\bigg|>\epsilon\right)\displaybreak[0]\\
\leq\frac{1}{\epsilon^2}\sum_{l=1}^\infty \frac{1}{k_{2^{l-1}}^2} E\max_{n=2^{l-1}+1,\ldots,2^{l}}\left(\sum_{i=1}^{k}\bigg(\frac{1}{p}\bigg\|\sum_{j\in B_i}\tilde{\Phi}_K(X_j)\bigg\|^2 -\frac{1}{p}E\bigg\|\sum_{j\in B_1}\tilde{\Phi}_K(X_j)\bigg\|^2\bigg)\right)^2\\
\leq C \frac{1}{\epsilon^2}\sum_{l=1}^\infty \frac{1}{k_{2^{l-1}}^2}k_{2^l}l^2 K^4\leq C\sum_{l=1}^\infty \frac{1}{k_{2^{l-1}}}l^2 K^4<\infty
\end{multline}
if we choose $K_l=2^{\eta l}$ and $\eta<\frac{c_1}{4}$, as we assumed $k_{2^{l-1}}\approx\frac{2^{l-1}}{p_{2^{l-1}}}\geq C2^{-lc_1}$. The Borel-Cantelli-lemma implies that $\tilde{I}_n\rightarrow0$ almost surely and thus $E^\star\left[\left\|S_{n1}^\star\right\|^2\right]\rightarrow E\left\|N_1\right\|^2$. This completes the proof.
\end{proof}

\end{document}